\documentclass[12pt, reqno]{amsart}
\usepackage[utf8]{inputenc}
\usepackage{amssymb,mathtools,cite,enumerate,color,eqnarray,hyperref,amsfonts,amsmath,amsthm,setspace,tikz,verbatim,charter,booktabs,multirow,xcolor}

\usepackage{youngtab}
\usepackage{tabularx}
\usepackage{ytableau}
\usetikzlibrary{tikzmark,shapes.geometric, arrows.meta, positioning, fit,decorations.pathreplacing}
\usepackage{subcaption}

\numberwithin{equation}{section}

\definecolor{ao(english)}{rgb}{0.0, 0.0, 0.6}

\hypersetup{colorlinks=true, linkcolor=ao(english),citecolor=ao(english)}

\usepackage[normalem]{ulem}

\theoremstyle{plain}
\newtheorem{theorem}{Theorem}
\numberwithin{theorem}{section}
\newtheorem{corollary}[theorem]{Corollary}
\newtheorem*{corollary*}{Corollary}
\newtheorem*{Example*}{Example}

\newtheorem{proposition}[theorem]{Proposition}

\newtheorem{conjecture}[theorem]{Conjecture}
\theoremstyle{definition}
\newtheorem{definition}[theorem]{Definition}
\newtheorem*{def*}{Definition}
\newtheorem*{theorem*}{Theorem}

\newtheorem*{definition*}{Definition}

\theoremstyle{remark}
\newtheorem*{remark}{Remark}

\allowdisplaybreaks

	\title[Hook Length Biases in $t$-Core Partitions]{Hook Length Biases in $t$-Core Partitions}

    \author[N. D. Baruah]{Nayandeep Deka Baruah}
\address{Department of Mathematical Sciences, Tezpur University, Assam, India, PIN-784028}
\email{nayan@tezu.ernet.in, nayandeeptezu@gmail.com}

	\author[H. Das]{Hirakjyoti Das}
	\address{Department of Mathematics, B. Borooah College (Autonomous), Guwahati 781007, Assam, India}
	\email{hirak@bborooahcollege.ac.in}

        \author[P. J. Mahanta]{Pankaj Jyoti Mahanta}
\address{Department of Mathematical Sciences, Tezpur University, Assam, India, PIN-784028}
\email{pjm2099@gmail.com, msp25007@tezu.ac.in}

\author[M. P. Saikia]{Manjil P. Saikia}
	\address{Mathematical and Physical Sciences Division, School of Arts and Sciences, Ahmedabad University, Ahmedabad 380009, Gujarat, India}
	\email{manjil.saikia@ahduni.edu.in}

	\keywords{$t$-Core partition; Combinatorics; Young Diagram; Hook length; Hook length biases}
	
	\subjclass[2020]{05A17; 05A20; 11P81; 11P82}
	
\begin{document}
	\begin{abstract}
		Recently, the theory of hook length biases has emerged as a prominent research topic. Led by Ballantine, Burson, Craig, Folsom, and Wen [\textit{Res. Math. Sci.}, 2023], hook length biases are being explored for ordinary partitions, odd versus distinct partitions, self-conjugate versus distinct odd partitions. Lately, Singh and Barman [\textit{J. Number Theory}, 2024] opened the door to hook length biases in $\ell$-regular partitions. In this work, we extend the theory of hook length biases to $t$-core partitions. For example, let $a_{t,k}(n)$ denote the number of hooks of length $k$ in all $t$-core partitions of $n$, then we find that $a_{3,1}(n)\ge a_{3,2}(n) \ge a_{3,4}(n)$ and $a_{4,1}(n)\ge a_{4,3}(n)$ for all $n$. The methods employed in this work are mainly combinatorial.
	\end{abstract}
	\maketitle
	
	\section{Introduction}
A partition of a positive integer $n$ is a sequence of positive integers $\lambda=(\lambda_1, \lambda_2, \ldots, \lambda_r)$ such that $\lambda_1\geq \lambda_2\geq \cdots \geq \lambda_r$ and $\sum\limits_{i=1}^r\lambda_i = n$. The numbers $\lambda_1, \lambda_2, \ldots, \lambda_r$ are called the parts of the partition $\lambda$.

A \textit{Young diagram} of a partition $(\lambda_1, \lambda_2, \ldots, \lambda_r)$ is a left-justified array of boxes, where the $i$-th row from the top contains $\lambda_i$ boxes. For example, the Young diagram of the partition $(6,3,2,1)$ is shown in Figure \ref{fig:hook} (A). The \textit{hook length} of a box in a Young diagram is the sum of the number of boxes directly right to it, the number of boxes directly below it, and 1 (for the box itself). For example, see Figure \ref{fig:hook} (B) for the hook length of each box in the Young diagram of the partition $(6,3,2,1)$. A hook of length $k$ is also called a $k$-hook.
\begin{figure}[h]
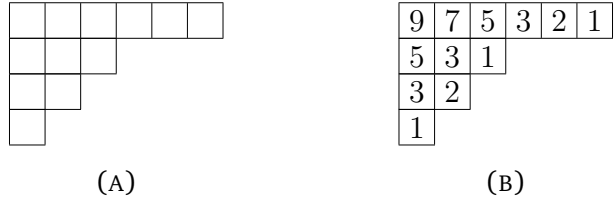

\centering
\begin{minipage}[b]{0.4\textwidth}
\[\young(~~~~~~,~~~,~~,~)\]
\subcaption{}
\end{minipage}
\begin{minipage}[b]{0.4\textwidth}
\[\young(975321,531,32,1)\]
\subcaption{}
\end{minipage}
\caption{The Young diagram of the partition $(6,3,2,1)$ and the hook length of each box}\label{fig:hook}
\end{figure}

For a positive integer $t\geq 2$, a $t$-core partition of $n$ is a partition of $n$ in which none of the hook lengths is divisible by $t$. For example, from Figure \ref{fig:hook} (B) we see that $(6,3,2,1)$ is a 4-core partition of 12.
    
Let $a_t(n)$ be the number of $t$-core partitions of $n$. Then it is known that \cite{garvan1990cranks}
    \begin{align*}
        \sum_{n=0}^\infty a_t(n)q^n&=\prod_{j=1}^\infty \dfrac{(1-q^{tj})^t}{1-q^j}.
    \end{align*}

One of the primary motivations for studying $t$-core partitions is their deep connection to the representation theory of symmetric groups. Specifically, $t$-core partitions determine the block structure of the representations of symmetric groups. For further aspects of $t$-core partitions in representation theory of symmetric groups, see \cite{GO96,JK81,Ols93}. Furthermore, $t$-core partitions also have a significant relation to the quadratic forms. This relationship is more fascinating when the above generating function is a weight-3/2 modular form. For example,  Ono and Sze \cite{OS97} proved that if $8n+5$ is square-free then $a_4(n)=1/2h(-32n-20)$, where $h(D)$ denotes the discriminant $D$ class number, viz. the order of the class group of the discriminant $D$ binary quadratic forms.

Now, let us recall another form of representation of a partition given by $$\lambda=(\lambda^{m_1}_1,\lambda^{m_2}_2,\ldots,\lambda^{m_r}_r),$$ where $m_i$ is the multiplicity of the part $\lambda_i$ and $\lambda_1>\lambda_2>\cdots>\lambda_r$.
When we consider multiplicity, the term `partition' refers exclusively to this representation.

Recently, inspired in part by hook-content formulas for specific restricted partitions in representation theory, Ballantine, Burson, Craig, Folsom, and Wen \cite{BBCFW23} investigated the total number of hooks of a given length in odd versus distinct partitions. Their work further motivated Singh and Barman \cite{SB24} to study hook length biases in ordinary and regular partitions. For integers $t \ge 2$ and $k \ge 1$, let $b_{t,k}(n)$ denote the number of hooks of length $k$ in all the $t$-regular partitions of $n$. Singh and Barman found that $b_{2,2}(n)\ge b_{2,1}(n)$ for all $n>4$, while $b_{2,2}(n)\ge b_{2,3}(n)$ for all $n\ge0$. They also conjectured that $b_{3,2}(n) \ge b_{3,1}(n)$ for all $n \ge 28$, which was first settled by He and Liu \cite{HL26}. He and Liu also proved that $b_{2,4}(n) \ge b_{2,3}(n)$ for all $n \ge 82$. Investigating another query by Singh and Barman, Qu and Zang \cite{QZ26} found that $b_{2,4}(n) \ge b_{2,5}(n)$ for all $n \ge 0$ and $n\neq 5$, and $b_{2,6}(n) \ge b_{2,7}(n)$ when $n\neq 7$.

In another work, Singh and Barman \cite{SB26} obtained $b_{t+1,1}(n) \ge b_{t,1}(n)$ for all $t\ge 2$ and $n \ge 0$. They also showed that $b_{3,2}(n) \ge b_{2,2}(n)$ for all $n \ge 3$, and $b_{3,3}(n) \ge b_{2,3}(n)$ for all $n \ge 0$. In the end, they conjectured that $b_{t+1,2}(n) \ge b_{t,2}(n)$ for all $t \ge 3$ and $n \ge 0$. Their conjecture, for the case $t=3$, was confirmed by Barman, Mahanta, and Singh \cite{BMS25}. Kim \cite{Kim25} partially confirmed the conjecture by proving that for each $\ell = 1, 2, 3$, there exist some
positive integers $N_t$ such that for all $n > N_t$, $b_{t+1,\ell}(n) \ge b_{t,\ell}(n)$. A complete proof was given by \cite{LZ25}.

Motivated by Ballantine, Burson, Craig, Folsom, and Wen \cite{BBCFW23}, Craig, Dawsey, and Han \cite{CDH26} compared the frequency with which partitions into odd parts and partitions into distinct parts have hook numbers equal to $h \ge 1$. They derived an asymptotic formula for the total number of hooks equal
to $h$ that appear among partitions into odd and distinct parts, respectively, to obtain their results. A similar work on self-conjugate partitions and partitions with distinct odd parts was done by Cossaboom \cite{Cos25}.

Continuing this line of research, we extend the theory of hook length biases to $t$-core partitions, noting that these partitions are themselves characterized by their hook lengths.

 Let $a_{t,k}(n)$ count the hooks with length $k$ in all the $t$-core partitions of $n$. We establish the following results for $a_{t,k}(n)$.
    
    \begin{proposition}\label{Thm 2 Core Formula}
For all $n\geq 0$, we have
\begin{align}
  \label{2coreformula}a_{2,2k+1}(n)=\begin{cases}
\ell-k, & \text{ if } n=\frac{\ell(\ell+1)}{2} \text{ for some non-negative}\\
& \text{ integer } \ell, \text{ where } 0\le k\le \ell-1,\\
 0, &    \text{ otherwise}.
\end{cases}  
\end{align}
\end{proposition}
    
    As an immediate consequence of the above proposition, we have the following bias among the hooks of odd lengths in 2-cores.
\begin{corollary}\label{Cor HLB 2 Core}
For all $\displaystyle{n=\dfrac{\ell(\ell+1)}{2}}$ with $\ell\ge 0$, we have
        \begin{align*}
            a_{2,2k+1}(n)-a_{2,2k+3}(n)=1,
        \end{align*}
        where $0\le k\le \ell-2$.
    \end{corollary}
    
    \begin{theorem}\label{Thm HLB 2-4}
        For all $n\ge 0$, we have
        \begin{align}
            a_{3,1}(n)\ge a_{3,2}(n) \ge a_{3,4}(n).
        \end{align}
    \end{theorem}

    \begin{theorem}\label{Thm HLB 4 Core 1-3}
        For all $n\ge 0$, we have
        \begin{align}
            \label{4 Core Bias 1}a_{4,1}(n)\ge a_{4,3}(n).
        \end{align}
    \end{theorem}

Numerical exploration in Mathematica suggests the following conjecture.

     \begin{conjecture}\label{Thm HLB 5 Core 1-3}
        For all $n\ge 0$, we have
        \begin{align}
             \label{5 Core Bias 1}a_{5,1}(n)\ge a_{5,3}(n)\ge a_{5,6}(n).
        \end{align}
    \end{conjecture}

    \begin{remark}
        It is worth noting that $a_{4,2}(n)$ does not always lie between $a_{4,1}(n)$ and $a_{4,3}(n)$ in \eqref{4 Core Bias 1}. For example, $a_{4,1}(4)=1<a_{4,2}(4)=2$, while $a_{4,2}(3)=2<a_{4,3}(3)=3$. 

        A similar phenomenon occurs for $a_{5,2}(n)$ and $a_{5,4}(n)$, 
        which do not necessarily lie between the pairs $a_{5,1}(n), a_{5,3}(n)$ and $a_{5,3}(n), a_{5,6}(n)$, respectively, in \eqref{5 Core Bias 1}.

        Therefore, it would be worthwhile to investigate the values of $n$ for which $a_{4,2}(n)$ does not lie between $a_{4,1}(n)$ and $a_{4,3}(n)$ in \eqref{4 Core Bias 1}. Similar investigations may also be performed for $a_{5,2}(n)$ and $a_{5,4}(n)$.
    \end{remark}

Let $a^{-C}_{t,k}(n)$ denote the total number of $k$-hooks in the $t$-core partitions of $n$ whose parts do not belong to the set $C$.

  \begin{theorem}\label{Thm HLB 4 Core no 1 and 2}
        For all $n\ge 0$, we have
        \begin{align}
            a^{-\{1,2\}}_{4,1}(n)= a^{-\{1,2\}}_{4,3}(n).
        \end{align}
    \end{theorem}

     \begin{theorem}\label{Thm HLB 4 Core no 1}
        For all $n\ge 0$, we have
        \begin{align}
            a^{-\{1\}}_{4,1}(n)\ge a^{-\{1\}}_{4,3}(n).
        \end{align}
    \end{theorem}
    
\begin{theorem}\label{Thm HLB 5 Core no 1 and 2}
        For all $n\ge 0$, we have
        \begin{align}
            a^{-\{1,2\}}_{5,1}(n)\le a^{-\{1,2\}}_{5,3}(n).
        \end{align}
    \end{theorem}

The following theorem arises as a natural extension, which presents biases among fixed hooks in different core partitions.

\begin{theorem}\label{Thm HLB 4-2 Core HL 1,3}
        For all $n\ge 0$ and $k\in\{1,3\}$, we have
        \begin{align}
            a_{2,k}(n)\le a_{4,k}(n).
        \end{align}
    \end{theorem}

    In the following sections, we prove our results.

    \section{Proof of Proposition \ref{Thm 2 Core Formula}}\label{sec: 2-core}
    It is well known that
        \begin{align}
            \label{Gen 2 Core}\sum_{n=0}^\infty a_2(n)q^n&=\prod_{j=1}^\infty \dfrac{(1-q^{2j})^2}{1-q^j}=\sum_{\ell=0}^\infty q^{\ell(\ell+1)/2}.
        \end{align}
        Therefore, $a_2(n)=0$ if $n\neq\ell(\ell+1)/2$ for any $\ell\geq 0$.

Combinatorially, it can be justified as follows. In the Young diagram of a partition, a 2-hook may arise in two different forms, as shown in Figure \ref{fig:types-2-hooks}.
\begin{figure}[h]
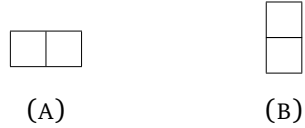

	\centering
	\begin{minipage}[b]{0.2\textwidth}
		\[\young(~~)\]
		\subcaption{}
	\end{minipage} \hspace{4mm}
	\begin{minipage}[b]{0.2\textwidth}
		\[\young(~,~)\]
		\subcaption{}
	\end{minipage}
	\caption{Types of 2-hooks}\label{fig:types-2-hooks}
\end{figure}

Therefore, in a 2-core partition, the multiplicity of any part cannot be greater than 1, and the difference between any two consecutive parts also cannot be greater than 1. Consequently, a partition containing no 2-hooks may take the form $(\ell, \ell-1,\ldots,2,1)$, for any $\ell\geq0$. Hence, $a_2(n)=0$ if $n\neq\ell(\ell+1)/2$ for any $\ell\geq 0$.

Again, when $\displaystyle{n=\dfrac{\ell(\ell+1)}{2}}$ for some $\ell\geq 0$, then the representation $(\ell, \ell-1,\ldots,2,1)$ of $n$ is unique. The nature of the staircase of the Young diagram of the partition $(\ell,\ell-1,\ldots,2,1)$ implies that the number of boxes directly right to any box is equal to the number of boxes directly below it. Consequently, the hook length of every box must be odd. That is, the only $2$-core partition of $\displaystyle{n=\dfrac{\ell(\ell+1)}{2}}$ is $(\ell,\ell-1,\ldots,2,1)$.

Note that different proofs of the above well known fact can be found in \cite{DS22,Rob00}.

Since the parts are consecutive, only the last box of each row in the Young diagram gives a hook of length 1. That gives $a_{2,1}(n)= \ell$ and proves the formula for $k=0$. Again, due to the consecutive parts of the unique 2-core partition, a hook of length 3 is located only between the last boxes of two consecutive rows in the Young diagram. There are $\ell-1$ such pairs of the last boxes. Therefore, $a_{2,3}(n)= \ell-1$, the formula for $k=1$. In the same way, a hook of length 5 occurs only between the last boxes of three consecutive alternate rows. In the Young diagram, there are $\ell-2$ pairs of such last boxes. Consequently, $a_{2,2k+1}(n)= \ell-k$, when $k=2$ is evident. This process can be repeated for the number of hooks for higher lengths to prove the other cases of the formula. The hook of the highest length $2\ell-1$ occurs between the boxes in the first and the last row in the Young diagram. So clearly, $a_{2,2\ell-1}(n)=a_{2,2(\ell-1)+1}(n)= 1=\ell-(\ell-1)$.
        
        Thus, the above arguments collectively prove \eqref{2coreformula}.
\qed

    \section{Combinatorial Proof of Theorem \ref{Thm HLB 2-4}}\label{sec: 3-core}
In the Young diagram of a partition, a 3-hook may arise in four different ways, as shown in Figure \ref{fig:types-3-hooks}.

\begin{figure}[h]
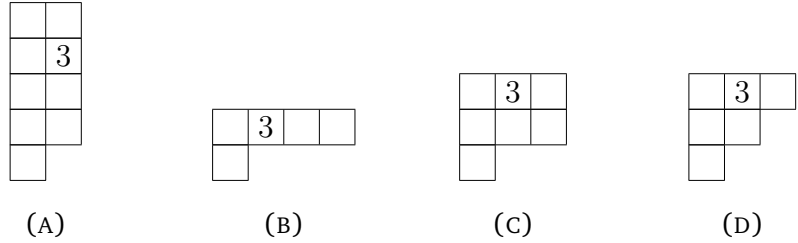

	\centering
	\begin{minipage}[b]{0.2\textwidth}
		\[\young(~~,~3,~~,~~,~)\]
		\subcaption{}
	\end{minipage} \hspace{4mm}
	\begin{minipage}[b]{0.2\textwidth}
		\[\young(~3~~,~)\]
		\subcaption{}
	\end{minipage}\hspace{4mm}
	\begin{minipage}[b]{0.2\textwidth}
		\[\young(~3~,~~~,~)\]
		\subcaption{}
	\end{minipage}\hspace{4mm}
	\begin{minipage}[b]{0.2\textwidth}
		\[\young(~3~,~~,~)\]
		\subcaption{}
	\end{minipage}
	\caption{Types of 3-hooks}\label{fig:types-3-hooks}
\end{figure}
Consequently, a $3$-core partition must satisfy the following conditions:
\begin{itemize}
    \item[3A:] $m_i\le 2$ for all $1\le i\le r$ (as implied by Type A 3-hook).
    \item[3B:] $\lambda_i-\lambda_{i+1}\le 2$ for all $1\le i\le r-1$, and and $\lambda_r\le 2$ (as implied by Type B 3-hook). 
\end{itemize}
Moreover, Type C and Type D 3-hooks, combined with Conditions 3A and 3B, imply that a $3$-core partition must satisfy the following conditions:
\begin{itemize}
    \item[3C:] If $m_i=2$, then $\lambda_i-\lambda_{i+1}=1$ for all $1\le i\le r-1$.
    \item[3D:] If $\lambda_i-\lambda_{i+1}=1 $, then $m_{i+1}= 2$ for all $1\le i\le r-1$.
\end{itemize}

Now, we eliminate the partitions that satisfy these conditions. Note, for example, that in the remaining partitions no 6-hooks occur. Three instances of 6-hooks are given in Figure \ref{fig:example-6-hooks}. The partitions containing these 6-hooks will be eliminated during the elimination of those with 3-hooks.

\begin{figure}[h]
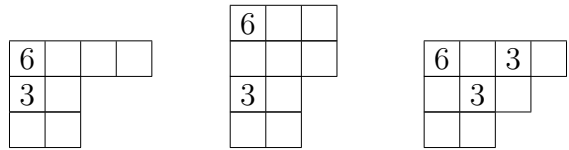

	\centering
	\begin{minipage}[b]{0.2\textwidth}
		\[\young(6~~~,3~,~~)\]
	\end{minipage}\hspace{0.5mm}
    \begin{minipage}[b]{0.2\textwidth}
		\[\young(6~~,~~~,3~,~~)\]
	\end{minipage}
    \hspace{0.5mm}
    \begin{minipage}[b]{0.2\textwidth}
		\[\young(6~3~,~3~,~~)\]
	\end{minipage}
	\caption{Three examples of 6-hooks}\label{fig:example-6-hooks}
\end{figure}

The Young diagram of any partition in the new set must take one of the forms shown in Figure \ref{fig:types-partitions}. Here, $\young(\vdots\vdots,\vdots\vdots)$ represents a tower of the form shown in Figure \ref{fig:tower}, of any height. Now, we will show that every partition of these forms is a 3-core partition.

\begin{figure}[h]
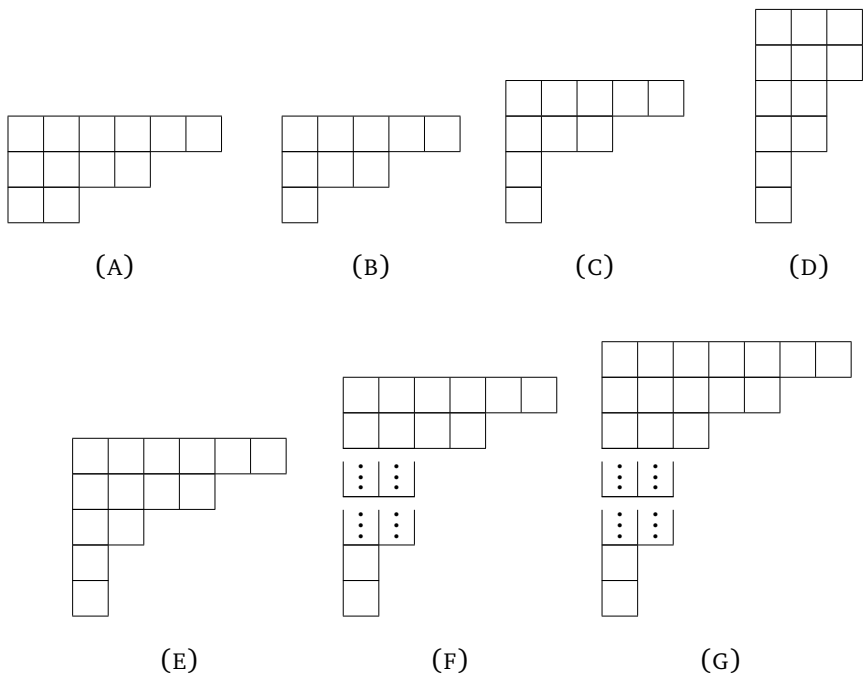

	\centering
	\begin{minipage}[b]{0.25\textwidth}
		\[\young(~~~~~~,~~~~,~~)\]
		\subcaption{}
	\end{minipage} \hspace{2mm}
	\begin{minipage}[b]{0.22\textwidth}
		\[\young(~~~~~,~~~,~)\]
		\subcaption{}
	\end{minipage}\hspace{2mm}
   	\begin{minipage}[b]{0.2\textwidth}
		\[\young(~~~~~,~~~,~,~)\]
		\subcaption{}
	\end{minipage}\hspace{2mm}
    \begin{minipage}[b]{0.2\textwidth}
		\[\young(~~~,~~~,~~,~~,~,~)\]
		\subcaption{}
	\end{minipage}\\
    
    \vspace{5mm}
    
	\begin{minipage}[b]{0.25\textwidth}
		\[\young(~~~~~~,~~~~,~~,~,~)\]
		\subcaption{}
	\end{minipage}  \hspace{2mm}
	\begin{minipage}[b]{0.25\textwidth}
		\[\young(~~~~~~,~~~~,\vdots\vdots,\vdots\vdots,~,~)\]
		\subcaption{}
	\end{minipage} \hspace{2mm}
	\begin{minipage}[b]{0.25\textwidth}
		\[\young(~~~~~~~,~~~~~,~~~,\vdots\vdots,\vdots\vdots,~,~)\]
		\subcaption{}
	\end{minipage}
	\caption{Types of partitions in the new set}\label{fig:types-partitions}
\end{figure}

\begin{figure}[h]
	\centering
		\[\young(~~~~,~~~~,~~~,~~~,~~,~~)\]
	\caption{A tower}\label{fig:tower}
\end{figure}

In the partitions of the form A in Figure \ref{fig:types-partitions}, we have $m_i=1$ for all $1\leq i\leq r$, and $\lambda_i-\lambda_{i+1}=2$ for all $1\leq i\leq r-1$, and $\lambda_r=2$. Therefore, if there are $\ell$ boxes directly below to a box, then there must be either $2\ell$ or $2\ell+1$ boxes directly right to it. Accordingly, every box has a hook length of either $3\ell+1$ or $3\ell+2$ for some $\ell \ge 0$. Hence, the partitions of the form A are 3-core partitions. A similar argument shows that the partitions of the forms B, C, D, E, F, and G are also 3-core partitions. 

Now, in Figure \ref{fig:types-partitions}, we observe that in the partitions of the form A, to each 1-hook there corresponds a 2-hook. Moreover, for $1\leq i\leq r-1$ to each 2-hook in $\lambda_i$ there corresponds a 4-hook. Therefore, here we have
    $$\#\text{ of 1-hooks } =\#\text{ of 2-hooks } =\#\text{ of 4-hooks }+1 .$$
Similarly, for the forms B, C, D, E, F, and G, we have
\begin{itemize}
    \item[]
    $\#\text{ of 1-hooks } =\#\text{ of 2-hooks } +1=\#\text{ of 4-hooks } +1,$
\item[]
    $\#\text{ of 1-hooks } =\#\text{ of 2-hooks } =\#\text{ of 4-hooks}+2,$
    \item[]
    $\#\text{ of 1-hooks }=\#\text{ of 2-hooks } =\#\text{ of 4-hooks}+1,$
    \item[]
    $\#\text{ of 1-hooks } =\#\text{ of 2-hooks }+1 =\#\text{ of 4-hooks}+1,$
    \item[]
    $\#\text{ of 1-hooks } =\#\text{ of 2-hooks } =\#\text{ of 4-hooks} +1,$
    \item[]
    $\#\text{ of 1-hooks } =\#\text{ of 2-hooks } +1 =\#\text{ of 4-hooks}+ 1,$
\end{itemize}
respectively.
Hence,
$$a_{3,1}(n)\ge a_{3,2}(n)\ge a_{3,4}(n),$$
for all $n\geq 0$. \qed

\section{Combinatorial proof of Theorem \ref{Thm HLB 4 Core 1-3} and Other Results}\label{sec: 4-core}

\begin{definition}
	The region of a hook in the Young diagram of a partition is defined as the shape formed by all boxes below the horizontal arm (or to the right of the vertical leg) of the hook, together with the boxes in the hook itself.
\end{definition}

\begin{figure}[h]
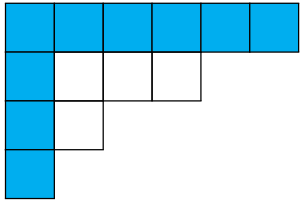
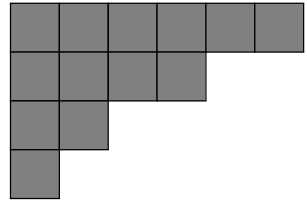

	\centering
	\begin{minipage}[b]{0.45\textwidth}
\begin{ytableau}
	*(cyan)~ & *(cyan)~ & *(cyan)~& *(cyan)~ & *(cyan)~ & *(cyan)~ \\
	*(cyan)~ & ~ & ~ & ~ \\
	 *(cyan)~ & ~\\
	 	 *(cyan)~
\end{ytableau}
	\subcaption{A 9-hook}
\end{minipage} \hspace{4mm}
\begin{minipage}[b]{0.45\textwidth}
\begin{ytableau}
	*(gray)~ & *(gray)~ & *(gray)~& *(gray)~ & *(gray)~ & *(gray)~ \\
	*(gray)~ & *(gray)~ & *(gray)~ & *(gray)~ \\
	*(gray)~ & *(gray)~\\
	*(gray)~
\end{ytableau}
		\subcaption{The region of the hook}
\end{minipage}
\caption{A hook and its region}\label{fig:region}
\end{figure}

\begin{theorem}\label{thm: region of hook}
The region of a $kt$-hook contains a $t$-hook, for all integers $t\geq 1$ and $k\geq 2$.
\end{theorem}

\begin{proof}
First, we prove that the region of a $2t$-hook contains a $t$-hook, for all integers $t\geq 1$.

Let us consider an arbitrary $2t$-hook, as shown in Figure \ref{fig:2t-hook}, where $1 \leq m \leq t$. Here, $\begin{ytableau}\cdots\end{ytableau}$ (resp., $\begin{ytableau}\vdots\end{ytableau}$) represents horizontal (resp., vertical) sequence of boxes whose total number is greater than or equal to 0. We denote the rows of the hook as $u_1, u_2, \ldots, u_{t-m+1}$, respectively, and the number of boxes in row $u_i$ within the region of the hook as $|u_i|$.

		\begin{figure}[!h]
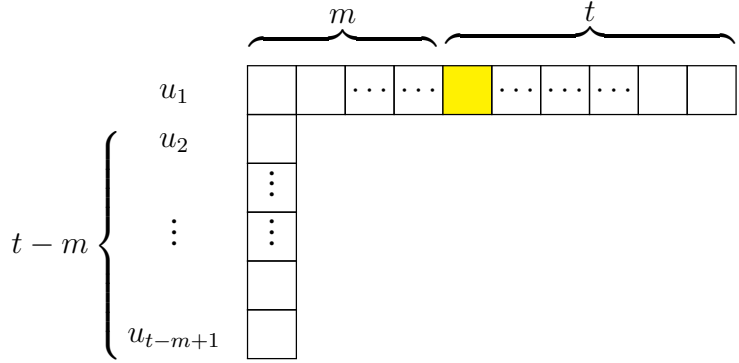

			\vspace{3em}
			\centering
			\begin{tabular}{r@{}l}
				\raisebox{-4.4em}{$t-m\left\{\vphantom{\begin{array}{c}~\\[4.5em] ~
					\end{array}}\right.$} &
				\hspace{2mm} 
				\begin{ytableau}
\none[u_1] & \none & ~    & ~ & \cdots & \cdots  & *(yellow) ~ & \cdots & \cdots &\cdots  & ~ & ~ \\
					\none[u_2] & \none & ~  \\
					\none & \none & \vdots   \\
					\none[\vdots] & \none & \vdots   \\
				\none & \none &	~   \\
				 \none[u_{t-m+1}] & \none &	~     
				\end{ytableau}\\[-11.3em]
  &\hspace{3.9em}$\overbrace{\hspace{5.9em}}^{\displaystyle m}$\\ [-2.1em]
	&\hspace{10.1em}$\overbrace{\hspace{9.1em}}^{\displaystyle t}$
			\end{tabular}\\ [9.2em]
			\caption{A $2t$-hook}\label{fig:2t-hook}
		\end{figure}
	
Now we try to construct a region of the hook so that no $t$-hook occurs. Therefore, $|u_2|$ is possibly $m+1$, as otherwise the yellow box in $u_1$ would have hook length $t$.

If $|u_2| = m+1$, then $u_2, u_3, \ldots, u_{t-m+1}$ form a $t$-hook. Consequently, $|u_2|$ is possibly $m+2$.

If $|u_2| = m+2$, then the rightmost $t-1$ boxes of $u_1$ and the rightmost box of $u_2$ form a $t$-hook. Consequently, $|u_3|$ is possibly $m+2$. But then, $u_3, \ldots, u_{t-m+1}$ form a $t$-hook. Therefore, $|u_3|$ is possibly $m+3$.

If $|u_3| = m+3$, then $u_2\geq m+3$. In this case, the rightmost $t-2$ boxes of $u_1$, one box of $u_2$, and the rightmost box of $u_3$ form a $t$-hook. Consequently, $|u_4|$ is possibly $m+3$.

Proceeding in this way, we see that $|u_{t-m+1}|$ is possibly $m+(t-m)=t$. Then, a $t$-hook must occur in $u_{t-m+1}$. That is, it is impossible to construct a region without the occurrence of a $t$-hook.

When $m\leq0$, we can apply the same process to the leg of the hook.

This completes the theorem for $k=2$.

For $k \geq 3$, the occurrence of a $t$-hook is even more clearly guaranteed. For example, in Figure \ref{fig:21-hook} a 21-hook is shown. To avoid the 7-hook in $u_1$, we see that $|u_2|$ must be at least 11. Similarly, if  $|u_2|=11$, to avoid the 7-hook in $u_2$, we see that $|u_3|$ is possibly 5. The subsequent arguments follow in a similar way to the case $k=2$.

	\begin{figure}[!h]
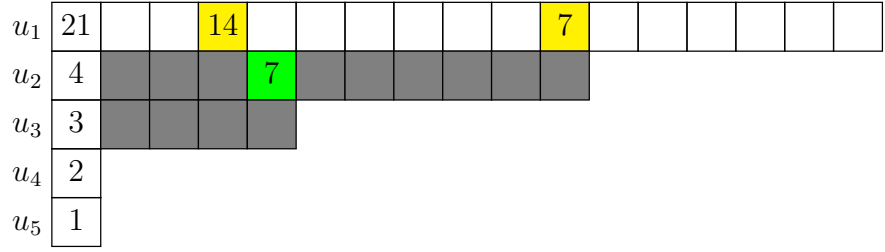

		\centering
		\begin{tabular}{r@{}l}
			\begin{ytableau}
				\none[u_1] & 21  & ~ & ~ & *(yellow) 14 & ~ & ~ & ~ & ~  & ~ & ~ & *(yellow) 7 & ~ & ~ & ~  & ~ & ~ & ~\\
				\none[u_2] & 4  & *(gray) ~ & *(gray) ~ & *(gray) ~ & *(green) 7 & *(gray) ~ & *(gray) ~ & *(gray) ~  & *(gray) ~ & *(gray) ~ & *(gray) ~ \\
				\none[u_3] &  3 & *(gray) ~  & *(gray) ~ & *(gray) ~ & *(gray) ~\\
				\none[u_4] & 2   \\
				\none[u_5] & 1     
			\end{ytableau}
		\end{tabular}
		\caption{A 21-hook}\label{fig:21-hook}
	\end{figure}
\end{proof}

Theorem \ref{thm: region of hook} implies that removing partitions of $n$ with a $t$-hook in their Young diagram leaves exactly the set of $t$-core partitions of $n$. As a result, applying Theorem \ref{thm: region of hook} simplifies the proof of Theorem \ref{Thm HLB 2-4}. Leaving this to the reader as an exercise, we now prove Theorem \ref{Thm HLB 4 Core 1-3} by applying Theorem \ref{thm: region of hook}.

\begin{proof}[Proof of Theorem \ref{Thm HLB 4 Core 1-3}]
In the Young diagram of a partition, the possible forms of 4-hooks are shown in Figure \ref{fig:types-4-hooks}.

\begin{figure}[h]
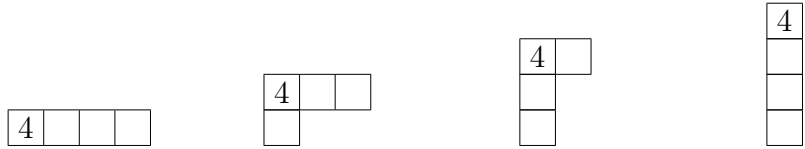

	\centering
	\begin{minipage}[b]{0.2\textwidth}
		\[\young(4~~~)\]
	\end{minipage} \hspace{4mm}
	\begin{minipage}[b]{0.2\textwidth}
		\[\young(4~~,~)\]
	\end{minipage} \hspace{4mm}
	\begin{minipage}[b]{0.2\textwidth}
		\[\young(4~,~,~)\]
	\end{minipage}\hspace{4mm}
	\begin{minipage}[b]{0.2\textwidth}
		\[\young(4,~,~,~)\]
	\end{minipage}
	\caption{Types of 4-hooks}\label{fig:types-4-hooks}
\end{figure}

Therefore, in a 4-core partition, we have
\begin{itemize}
    \item $m_i\leq 3$ and $\lambda_i-\lambda_{i+1}\leq 3$ for all $1\leq i\leq r$ and we take $\lambda_{r+1}=0$.
    \item If $\lambda_i-\lambda_{i+1}= 3$, then $m_i=1$ and $m_{i+1}=1, 2, \text{ or } 3$.
    \item If $\lambda_i-\lambda_{i+1}= 2$, then
    \begin{itemize}
        \item when $m_i=1$, then $m_{i+1}=2, \text{ or } 3$,
        \item when $m_i=2$, then $m_{i+1}=2, \text{ or } 3$,
        \item $m_i\neq 3$.
    \end{itemize}
    \item If $\lambda_i-\lambda_{i+1}= 1$, then
    \begin{itemize}
        \item when $m_i=1$, then $m_{i+1}=1, \text{ or } 3$,
        \item when $m_i=2$, then $m_{i+1}=3$,
        \item $m_i=3=m_{i+1}$.
    \end{itemize}
\end{itemize}

The base step of the staircase of the Young diagram of a 4-core partition takes one of the forms shown in Figure \ref{fig:base-steps-4-core}. In the forms D, I, J, M, O, and P in Figure \ref{fig:base-steps-4-core}, the number of 1-hooks is equal to the number of 3-hooks. In the remaining forms, the number of 1-hooks is greater than the number of 3-hooks.

\begin{figure}[h]
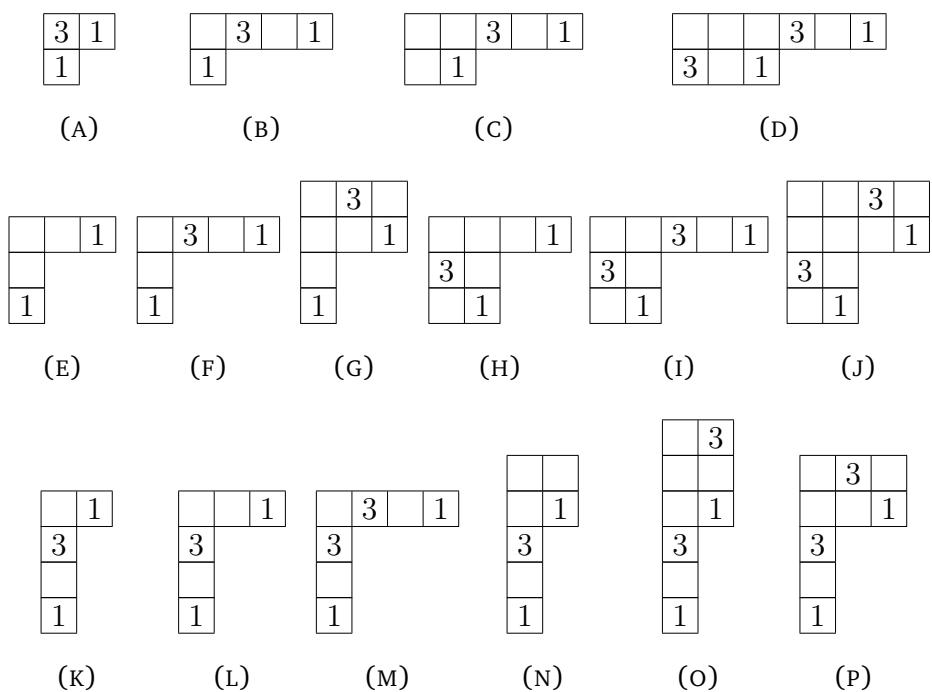

	\centering
\begin{minipage}[b]{0.15\textwidth}
		\[\young(31,1)\]
        \subcaption{}
	\end{minipage} \hspace{1mm}
	\begin{minipage}[b]{0.18\textwidth}
		\[\young(~3~1,1)\]
        \subcaption{}
	\end{minipage}\hspace{2.5mm}
    \begin{minipage}[b]{0.25\textwidth}
		\[\young(~~3~1,~1)\]
        \subcaption{}
	\end{minipage}\hspace{2mm}
\begin{minipage}[b]{0.3\textwidth}
		\[\young(~~~3~1,3~1)\]
        \subcaption{}
	\end{minipage}\\
    
    \vspace{2mm}
    
	\begin{minipage}[b]{0.13\textwidth}
		\[\young(~~1,~,1)\]
        \subcaption{}
	\end{minipage}\hspace{0.5mm}
    \begin{minipage}[b]{0.15\textwidth}
		\[\young(~3~1,~,1)\]
        \subcaption{}
	\end{minipage}\hspace{0.5mm}
    \begin{minipage}[b]{0.13\textwidth}
		\[\young(~3~,~~1,~,1)\]
        \subcaption{}
	\end{minipage}\hspace{0.5mm}
    \begin{minipage}[b]{0.15\textwidth}
		\[\young(~~~1,3~,~1)\]
        \subcaption{}
	\end{minipage}\hspace{0.5mm}
    \begin{minipage}[b]{0.2\textwidth}
		\[\young(~~3~1,3~,~1)\]
        \subcaption{}
	\end{minipage}\hspace{0.5mm}
    \begin{minipage}[b]{0.15\textwidth}
		\[\young(~~3~,~~~1,3~,~1)\]
        \subcaption{}
	\end{minipage}\\
    
    \vspace{2mm}

    \begin{minipage}[b]{0.15\textwidth}
		\[\young(~1,3,~,1)\]
        \subcaption{}
	\end{minipage}\hspace{0.5mm}
    \begin{minipage}[b]{0.15\textwidth}
		\[\young(~~1,3,~,1)\]
        \subcaption{}
	\end{minipage}\hspace{0.5mm}
    \begin{minipage}[b]{0.15\textwidth}
		\[\young(~3~1,3,~,1)\]
        \subcaption{}
	\end{minipage}\hspace{0.5mm}
    \begin{minipage}[b]{0.15\textwidth}
		\[\young(~~,~1,3,~,1)\]
        \subcaption{}
	\end{minipage}\hspace{0.5mm}
    \begin{minipage}[b]{0.15\textwidth}
		\[\young(~3,~~,~1,3,~,1)\]
        \subcaption{}
	\end{minipage}\hspace{0.5mm}
    \begin{minipage}[b]{0.15\textwidth}
		\[\young(~3~,~~1,3,~,1)\]
        \subcaption{}
	\end{minipage}
	\caption{Forms of the base step in the Young diagram of a 4-core partition}\label{fig:base-steps-4-core}
\end{figure}

Now we consider the following five cases.

\noindent \textbf{Case I: The horizontal arm of the base step is $\young(~)$. (The forms are A and K.)} In this case, the horizontal arm in the next step must be either $\young(~)$ or $\young(~~~)$. In both cases, the numbers of 1-hooks and 3-hooks each increase by 1.

\noindent \textbf{Case II: The horizontal arm of the base step is $\young(~~)$ and $\young(~~~)$. (The forms are B, C, D, E, F, H, I, L, and K.)} In this case, the horizontal arm in the next step must be $\young(~~~)$. Then, the numbers of 1-hooks and 3-hooks each increase by 1. Therefore, up to this step, the total number of 1-hooks equals the total number of 3-hooks for forms D, I, and M, whereas for the remaining forms, the total number of 1-hooks exceeds the total number of 3-hooks.

\noindent \textbf{Case III: The horizontal arm of the base step is $\young(~,~)$. (The form is N.)} Then, the horizontal arm in the next step must be $\young(~~)$, $\young(~~~)$, or $\young(~~,~~)$. Here, in the first case, the number of 1-hooks increases by 1 while the number of 3-hooks remains unchanged. In the latter two cases, the number of 1-hooks and 3-hooks each increase by 1.

\noindent \textbf{Case IV: The horizontal arm of the base step is $\young(~,~,~)$. (The form is O.)} In this case, the horizontal arm in the next step must be $\young(~)$, $\young(~~)$, $\young(~,~)$, $\young(~~~)$, $\young(~,~,~)$, or $\young(~~,~~)$. In the first three cases, the number of 1-hooks increases by 1 while the number of 3-hooks remains unchanged. In the latter three cases, the number of 1-hooks and 3-hooks each increase by 1.

\noindent \textbf{Case V: The horizontal arm of the base step is $\young(~~,~~)$. (The forms are G, J, and P.)} Here, the horizontal arm in the next step must be $\young(~~)$, $\young(~~~)$, or $\young(~~,~~)$. When the arm is $\young(~~)$, then in each form the number of 1-hooks increases by 1 while the number of 3-hooks remains unchanged. Again, When the arm is $\young(~~~)$ or $\young(~~,~~)$, then in each form, the numbers of 1-hooks and 3-hooks both increase by 1. 

Thus, by counting the 1-hooks and 3-hooks step by step in the Young diagram of a 4-core partition, we find that the total number of 1-hooks is always greater than or equal to the total number of 3-hooks.

This completes the proof.
\end{proof}

\begin{remark}
    The statement of Theorem \ref{Thm HLB 4 Core 1-3} concerns the total number of 1-hooks and 3-hooks in all 4-core partitions of $n$. But, the proof demonstrates that for each 4-core partition of $n$, the total number of 1-hooks is greater than or equal to the total number of 3-hooks.
\end{remark}

There are forty possible forms of the base step of the staircase of the Young diagram of a 5-core partition. Listing them as in the proof of Theorem \ref{Thm HLB 4 Core 1-3} is not a viable method to prove Conjecture \ref{Thm HLB 5 Core 1-3}. However, Theorem \ref{Thm HLB 5 Core no 1 and 2} can be proved using an analogous approach. We now proceed to the proofs of Theorem \ref{Thm HLB 4 Core no 1 and 2}, Theorem \ref{Thm HLB 4 Core no 1}, and Theorem \ref{Thm HLB 5 Core no 1 and 2}.

\begin{proof}[Proof of Theorem \ref{Thm HLB 4 Core no 1 and 2}]
The form of the base step in the Young diagram of a 4-core partition when the smallest part is not 1 and 2 is D as shown in Figure \ref{fig:base-steps-4-core}. Here, the number of 1-hooks is equal to the number of 3-hooks. In this form, the horizontal arm in the next step must be $\young(~~~)$. Then, the numbers of 1-hooks and 3-hooks each increase by 1. This completes the proof.
\end{proof}

\begin{proof}[Proof of Theorem \ref{Thm HLB 4 Core no 1}]
When the smallest part is not 1, the forms of the base step in the Young diagram of a 4-core partition must take one of the forms C, D, H, I, or J as shown in Figure \ref{fig:base-steps-4-core}. For the form J, the horizontal arm in the next step must be $\young(~~)$, $\young(~~~)$, or $\young(~~,~~)$. Moreover, for the remaining forms, it must be $\young(~~~)$. The remainder of the argument is left to the reader.
\end{proof}

\begin{proof}[Proof of Theorem \ref{Thm HLB 5 Core no 1 and 2}]
    The forms of the base step in the Young diagram of a 5-core partition when the smallest part is not 1 and 2 are shown in Figure \ref{fig:base-steps-5-core-smallest-part-3}. In each form, the number of 1-hook is less than or equal to the number of 3-hooks.

\begin{figure}[h]
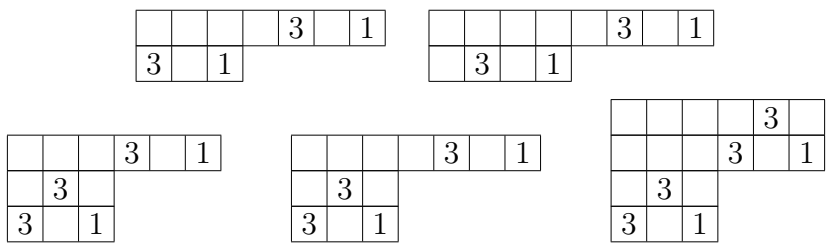

	\centering
\begin{minipage}[b]{0.3\textwidth}
		\[\young(~~~~3~1,3~1)\]
	\end{minipage} \hspace{1mm}
	\begin{minipage}[b]{0.3\textwidth}
		\[\young(~~~~~3~1,~3~1)\]
	\end{minipage}\\
    \begin{minipage}[b]{0.3\textwidth}
		\[\young(~~~3~1,~3~,3~1)\]
	\end{minipage}\hspace{1mm}
\begin{minipage}[b]{0.3\textwidth}
		\[\young(~~~~3~1,~3~,3~1)\]
	\end{minipage}\hspace{1mm}    
	\begin{minipage}[b]{0.3\textwidth}
		\[\young(~~~~3~,~~~3~1,~3~,3~1)\]
	\end{minipage}
	\caption{Forms of the base step in the Young diagram of a 5-core partition when the smallest part is not 1 and 2}\label{fig:base-steps-5-core-smallest-part-3}
\end{figure}

For the first four forms, the horizontal arm in the next step must be $\young(~~~~)$. Then, the numbers of 1-hooks and 3-hooks each increase by 1.

Moreover, for the last form, the horizontal arm in the next step must be $\young(~~~)$, $\young(~~~~)$, or $\young(~~~,~~~)$. In the first two cases, the numbers of 1-hooks and 3-hooks each increase by 1, and in the last case, the number of 3-hooks exceeds the number of 1-hooks.

This completes the proof.
\end{proof}

\begin{remark}
    Figure \ref{fig:base-steps-4-core} or the list of forty possible forms of the base step of the staircase of the Young diagram of a 5-core partition can provide more information. For example,

$$ a^{-\{1,2\}}_{4,1}(n)= a^{-\{1,2\}}_{4,3}(n)=
\begin{cases}
    \ell, &\text{ if } n=\frac{3\ell(\ell+1)}{2} \text{ for some positive integer } \ell,\\
    0, & \text{ otherwise}.
\end{cases}$$
\end{remark}

\section{Proof of Theorem \ref{Thm HLB 4-2 Core HL 1,3}}
From Proposition \ref{2coreformula}, we know that $a_{2,1}(n)=\ell$, $a_{2,3}(n)=\ell -1$ for $n=\ell(\ell+1)/2$. To prove Theorem \ref{Thm HLB 4-2 Core HL 1,3}, it therefore suffices to show that $a_{4,3}(n)\ge\ell$ for $n=\ell(\ell+1)/2$. The form A in Figure \ref{fig:base-steps-4-core} in Section \ref{sec: 4-core} implies $a_{4,3}(n)\ge\ell-1$. Consequently, the task reduces to proving the existence of at least two 4-core partitions of $n$, whenever $n=\ell(\ell+1)/2$. The following alternative, yet elementary, approach settles the problem. Note that 
 \begin{align*}
     \sum_{n=0}^\infty a_4(n)q^n&=\prod_{j=1}^\infty \dfrac{(1-q^{4j})^4}{1-q^j}=\prod_{j=1}^\infty\dfrac{(1-q^{2j})^2}{1-q^j}\cdot \dfrac{(1-q^{4j})^2}{1-q^{2j}}\cdot \dfrac{(1-q^{4j})^2}{1-q^{2j}}.
 \end{align*}
 Due to the second equality in \eqref{Gen 2 Core}, we have
 \begin{align*}
     \sum_{n=0}^\infty a_4(n)q^n&=\sum_{m,r,s=0}^\infty q^{m(m+1)/2+r(r+1)+s(s+1)}\\
     &=\sum_{m=0}^\infty q^{m(m+1)/2}+\sum_{\substack{m,r,s=0\\\textup{$r, s\neq 0$ simultaneously}}}^\infty q^{m(m+1)/2+r(r+1)+s(s+1)}.
 \end{align*}
 In the above identity, we see exponents of the form $m(m+1)/2$ in the first summation. This means that the first summation contributes one 4-core partition for every triangular number. Therefore, the objective reduces to confirming that for every $h\ge 2$, there are integers $m,r,s$ such that
 \begin{align*}
     \dfrac{h(h+1)}{2}&=\dfrac{m(m+1)}{2}+r(r+1)+s(s+1),
 \end{align*}equivalently,
 \begin{align*}
     (2h+1)^2+4&=(2m+1)^2+2(2r+1)^2+2(2s+1)^2.
 \end{align*}

Dickson \cite[Theorem VI]{Dic27} proved that the quadratic form $x^2+2y^2+2z^2$ represents all positive integers not of the form $4^k(8n+7)$.

Now, for any $h\ge 2$, we have
\[
    (2h+1)^2 + 4\equiv 5 \pmod 8,
    \]
which readily implies that $(2h+1)^2+4$ is not of the form $4^k(8n + 7)$, as $4^k(8n+7)\equiv 0 \textup{~or~} 4 \pmod8.$ Therefore, by Dickson's result \cite[Theorem VI]{Dic27}, there exist integers $x, y, z$ such that $(2h+1)^2 + 4 = x^2 + 2y^2 + 2z^2$. It remains to show that $x, y$, and $z$ are odd. From the above arguments, we have
\begin{equation}\label{eq1}
        x^2 + 2y^2 + 2z^2 \equiv 5 \pmod 8.
\end{equation}
Since $2y^2 + 2z^2$ is even, $x^2$ must be odd, which implies that $x$ is odd. Thus, equation \eqref{eq1} reduces to
\[
    y^2 + z^2 \equiv 2 \pmod 4.
\]
Now it is routine that a case by case analysis of the parities of $y$ and $z$ implies that the only way the above holds is when $y$ and $z$ are both odd. Thus, we have proved that for every $h\ge 2$ there exist integers $m, r, s$ such that
\[
(2h+1)^2+4 = (2m+1)^2+2(2r+1)^2+2(2s+1)^2.
\]
This completes the proof of the theorem. \qed

	\section*{Statements and Declarations}

	\textbf{Data Availability} This manuscript does not contain any associated data.\\
	
	\textbf{Competing Interests} The authors declare that they have no competing interests.\\
	
	\textbf{Funding Information}
    The third author was partially supported by an institutional fellowship for doctoral research from Tezpur University, Assam, India.


\end{document}